\documentclass[12pt,a4paper]{article} 

\usepackage{url}
\usepackage{chngcntr}
\usepackage{amsmath}
\usepackage{amssymb}
\usepackage{amsthm}
\usepackage{setspace}
\usepackage{xypic}
\usepackage{graphicx}
\usepackage[nottoc,notlot,notlof]{tocbibind}
\usepackage{grffile}

\usepackage{authblk}
\usepackage[top=2cm, bottom=2cm, left=2cm, right=2cm]{geometry}

\input xy
\xyoption{all}

\newtheorem{prop}{Proposition}[section]
\newtheorem{lem}[prop]{Lemma}
\newtheorem{theorem}[prop]{Theorem}
\newtheorem*{theorem*}{Theorem}
\theoremstyle{definition}
\newtheorem{defu}[prop]{Definition}
\newtheorem*{defu*}{Definition}
\newtheorem{rem}[prop]{Remark}

\newcommand{\st}{_\textnormal{st}}
\newcommand{\hr}{_\textnormal{hr}}
\newcommand{\lr}{_\textnormal{lr}}

\begin{document}
\pagenumbering{arabic}
\setcounter{page}{1}

\title{On the restrictiveness of the hazard rate order}

\author{Sela Fried \thanks{This work was done while the author was a postdoctoral fellow in the Department of Computer Science at the Ben-Gurion University of the Negev.} }
\date{} 
\maketitle
\begin{abstract}
Every element $\theta=(\theta_1,\ldots,\theta_n)$ of the probability $n$-simplex induces a probability distribution $P_\theta$ of a random variable $X$ that can assume only a finite number of real values $x_1 < \cdots < x_n$ by defining $P_\theta(X=x_i) = \theta_i, 1\leq i \leq n$. We show that if $\Theta$ and $\Theta'$ are two random vectors uniformly distributed on $\Delta^n$, then $P(P_\Theta\leq\hr P_{\Theta'})=\frac{1}{2^{n-1}}$ where $\leq\hr$ denotes the hazard rate order.
\end{abstract} 

\section{Introduction}

Stochastic orders are partial orders that are used to compare probability distributions and have applications in diverse areas of probability and statistics. Since they are, in general, only partial orders, it may well happen that two given probability distributions are not comparable with respect to some stochastic order of interest. Although they are well studied (see, for example, \cite{MS} or \cite{Sh}), to the best of our knowledge, the question of their restrictiveness has not yet been addressed, in terms of how likely it is for two randomly chosen probability distributions to be comparable with respect to a certain stochastic order. In this work we answer this question for probability distributions of random variables that can assume only a finite number of real numbers and for an important and common stochastic order: the hazard rate order, denoted by $\leq\hr$. 

More precisely, consider a vector $\theta=(\theta_1,\ldots,\theta_n)$ of real numbers such that $\theta_1,\ldots,\theta_n\geq 0$ and $\theta_1 + \cdots + \theta_n=1$. Thus, $\theta$ is an element of the probability $n$-simplex $\Delta^n$. It induces a probability distribution $P_\theta$ of a random variable $X$ that can assume only a finite number of real numbers $x_1<\cdots<x_n$ by defining $P_\theta(X = x_i) = \theta_i, \;1\leq i\leq n$. Our main result is  

\begin{theorem*} 
Let $n\in\mathbb{N}$ and let $\Theta$ and $\Theta'$ be two independent random vectors uniformly distributed on the probability $n$-simplex $\Delta^n$. Then
$$P(P_\Theta\leq\hr P_{\Theta'})=\frac{1}{2^{n-1}}.$$
\end{theorem*}

This work is a continuation of \cite{Fr} where it was shown that, in the notation of the above theorem, $$P(P_\Theta\leq\st P_{\Theta'})=\frac{1}{n}\;\; \textnormal{ and }\;\; P(P_\Theta\leq\lr P_{\Theta'})=\frac{1}{n!}$$ where $\leq\st$ and $\leq\lr$ denote the usual stochastic order and the likelihood ratio order, respectively.

\section{Preliminaries}

In this work, $n\geq 2$ is a natural number and $u$ a positive real number. 

\begin{defu}\label{def; simplex}
We denote
$$ \Delta^{n,u} =   \{(\theta_1,\ldots,\theta_n)\in\mathbb{R}^n\; |\; 
  \theta_1 + \cdots + \theta_n = u,\;  \theta_i \geq 0, \;1\leq i \leq n\}.$$
If $u=1$ then $\Delta^n = \Delta^{n,1}$ is merely the probability $n$-simplex.
\end{defu}

\begin{lem}\label{lem 2}
The volume $\textnormal{Vol}(\Delta^{n,u})$ of $\Delta^{n,u}$ is $\frac{\sqrt{n}u^{n-1}}{(n-1)!}$.
\end{lem}
\begin{proof}
By \cite{Ell}, the volume of the set $$
\Sigma^{n-1,u} =  \{(\theta_1,\ldots,\theta_{n-1})\in\mathbb{R}^{n-1}\; |\; 
  \theta_1 + \cdots + \theta_{n-1}\leq u,\;  \theta_i \geq 0, \;1\leq i \leq n-1\}
$$ is $\frac{u^{n-1}}{(n-1)!}$. By \cite[Theorem on p. 13]{Jon}, $\textnormal{Vol}(\Delta^{n,u})=\sqrt{n}\textnormal{Vol}(\Sigma^{n-1,u})$.
\end{proof}

\begin{defu}\label{def; par} 
Let $\preceq$ be a partial order on $\Delta^{n,u}, \theta\in\Delta^{n,u}$. We denote $$\Delta^{n,u}_{\succeq \theta} = \{\theta'\in\Delta^{n,u}\;|\;\theta'\succeq \theta\} .$$
\end{defu}

\section{The hazard rate order}

The following definition is a modification of the definition in \cite[1.B.10 on p. 17]{Sh}:

\begin{defu}\label{def; fosd}
Let $\theta=(\theta_1,\ldots,\theta_n), \theta'=(\theta'_1,\ldots,\theta'_n)\in \Delta^{n,u}$. We say that \textbf{$\theta$ is smaller than $\theta'$ in the hazard rate order} and write $\theta\leq\hr \theta'$ if $\Big(\sum_{k=i}^n \theta_k\Big)\Big(\sum_{k=j}^n \theta'_k\Big)\geq\Big(\sum_{k=j}^n \theta_k\Big)\Big(\sum_{k=i}^n \theta'_k\Big)$ for each $1\leq i\leq j\leq n$.
\end{defu}

The following lemma shows that given the last coordinate, the hazard rate order can be verified in one dimension less. Its proof is an easy exercise.

\begin{lem}\label{lem 1}
Let $\theta=(\theta_1,\ldots,\theta_{n+1}), \theta'=(\theta'_1,\ldots,\theta'_{n+1})\in\Delta^{n+1,u}$. Suppose $\theta'_{n+1}<u$. Then $\theta\leq\hr \theta'  \;(\text{in } \Delta^{n+1,u})$ if and only if $\theta'_1\leq\theta_1$ and
$$
(\frac{\theta_2}{v}, \ldots, \frac{\theta_{n+1}}{v}) \leq\hr  (\theta'_2,\ldots,\theta'_{n+1})\;( \text{in }\Delta^{n,u-\theta'_{n+1}}) \text{ where } v = \frac{\sum_{k=2}^{n+1}\theta_k}{u-\theta'_1}.
$$
\end{lem}

\begin{lem}\label{lem 3}
Let $\Theta$ be a random variable uniformly distributed on $\Delta^{n,u}$ and let $\theta=(\theta_1,\ldots,\theta_n)\in\Delta^{n,u}$. Then 
$$P(\Theta\geq\hr \theta)=\prod_{i=1}^{n-1}\frac{\Big(\sum_{j=i}^n \theta_j\Big)^{n-i}-\Big(\sum_{j=i+1}^n \theta_j\Big)^{n-i}}   {\Big(\sum_{j=i}^n \theta_j\Big)^{n-i}} .$$
\end{lem}

\begin{proof}
We proceed by induction. Suppose the claim holds for $n$ and let $\theta=(\theta_1,\ldots,\theta_{n+1})\in\Delta^{n+1,u}$. Then 
\begin{align}
P(\Theta\geq\hr \theta) = &\frac{n!}{u^n}\int_0^{\theta_1}\frac{(u-x_1)^{n-1}}{(n-1)!}P(\Theta'\geq\hr(\theta_2/v,\ldots,\theta_{n+1}/v))dx_1 \nonumber \\ =&\frac{1}{u^n}\prod_{i=1}^{n-1}\frac{\Big(\sum_{j=i+1}^{n+1} \theta_j\Big)^{n-i}-\Big(\sum_{j=i+2}^{n+1} \theta_j\Big)^{n-i}} {\Big(\sum_{j=i+1}^{n+1} \theta_j\Big)^{n-i}}(u^n-(u-\theta_1)^n) \nonumber \\ = & \prod_{i=1}^n\frac{\Big(\sum_{j=i}^{n+1} \theta_j\Big)^{n+1-i}-\Big(\sum_{j=i+1}^{n+1} \theta_j\Big)^{n+1-i}} {\Big(\sum_{j=i}^{n+1} \theta_j\Big)^{n+1-i}}\nonumber.
\end{align}
\end{proof}
We come to the main result of this work:
\begin{theorem}
Let $\Theta$ and $\Theta'$ be two independent random vectors uniformly distributed on $\Delta^{n,u}$. Then $$P(\Theta \leq\hr \Theta')=\frac{1}{2^{n-1}}.$$
\end{theorem}
\begin{proof}
It holds 
\begin{align}
P(\Theta \leq\hr \Theta') = & \frac{(n-1)!}{\sqrt{n}u^{n-1}}\int_{\Delta^{n,u}} \prod_{i=1}^{n-1}\frac{\Big(\sum_{j=i}^n \theta_j\Big)^{n-i}-\Big(\sum_{j=i+1}^n \theta_j\Big)^{n-i}}   {\Big(\sum_{j=i}^n \theta_j\Big)^{n-i}}d(\theta_1,\ldots,\theta_n)\nonumber \\ =& \frac{(n-1)!}{u^{n-1}}\int_0^u\int_0^{u-\theta_2} \cdots\int_0^{u-\sum_{i=2}^{n-1}\theta_i} \prod_{i=1}^{n-1}\frac{\Big(\sum_{j=i}^n \theta_j\Big)^{n-i}-\Big(\sum_{j=i+1}^n \theta_j\Big)^{n-i}}   {\Big(\sum_{j=i}^n \theta_j\Big)^{n-i}} d\theta_n\cdots d\theta_2 \nonumber  \\
=& \frac{(n-1)!}{u^{n-1}}\int_0^u\int_0^{u-\theta_2} \cdots\int_0^{u-\sum_{i=2}^{n-1}\theta_i}\frac{\Big(\sum_{j=1}^n\theta_j\Big)^{n-1}-\Big(\sum_{j=2}^n\theta_j\Big)^{n-1}}{\Big(\sum_{j=1}^n\theta_j\Big)^{n-1}} \cdot \nonumber \\ & \hspace{5.5cm}\prod_{i=2}^{n-1}\frac{\Big(\sum_{j=i}^n \theta_j\Big)^{n-i}-\Big(\sum_{j=i+1}^n \theta_j\Big)^{n-i}}   {\Big(\sum_{j=i}^n \theta_j\Big)^{n-i}} d\theta_n\cdots d\theta_2 \nonumber  \\ =& \frac{(n-1)!}{u^{n-1}}\int_0^u\int_0^{u-\theta_2} \cdots\int_0^{u-\sum_{i=2}^{n-1}\theta_i}\prod_{i=2}^{n-1}\Bigg(1-\frac{\Big(\sum_{j=i+1}^n \theta_j\Big)^{n-i}}{\Big(\sum_{j=i}^n \theta_j\Big)^{n-i}}\Bigg) d\theta_n\cdots d\theta_2 \nonumber \\ &\hspace{1cm}- \frac{(n-1)!}{u^{2n-1}}\int_0^u\int_0^{u-\theta_2} \cdots\int_0^{u-\sum_{i=2}^{n-1}\theta_i}\Big(\Big(\sum_{j=2}^n \theta_j\Big)^{n-1}-\Big(\sum_{j=2}^n \theta_j\Big)\Big(\sum_{j=3}^n \theta_j\Big)^{n-2}\Big)\cdot \nonumber \\& \hspace{5.5cm} \prod_{i=3}^{n-1}\Bigg(1-\frac{\Big(\sum_{j=i+1}^n \theta_j\Big)^{n-i}}{\Big(\sum_{j=i}^n \theta_j\Big)^{n-i}}\Bigg) d\theta_n\cdots d\theta_2  \label{aa}
\end{align}
Consider the following substitution, which is a variation of \cite[Exercise 9.13.1]{Shu}: $$\theta_i=
\begin{cases}
    (1-y_{i+1}) \prod_{j=2}^iy_j,& 2\leq i \leq n-1\\
    \prod_{j=1}^ny_j,              & i=n
\end{cases}.$$ It is easily verified that the Jacobian is given by $\prod_{i=2}^{n-1} y_i^{n-i}$ and that $$\prod_{j=2}^iy_j=\sum_{j=i}^n \theta_j, \;\;2\leq i\leq n.$$ Thus, 

\begin{align}
(\ref{aa})  = &\frac{(n-1)!}{u^{n-1}}\int_0^u y_2^{n-2}dy_2\prod_{i=3}^{n-1}\int_0^1(1-y_{i}^{n+1-i})y_i^{n-i}dy_i\int_0^1 1-y_ndy_n \nonumber \\ &\hspace{1cm}- \frac{(n-1)!}{u^{2n-1}}\int_0^u\int_0^1y_2^{2n-3}y_3^{n-3}-y_2^{2n-3}y_3^{2n-5}dy_3dy_2\cdot \nonumber \\& \hspace{5.5cm} \prod_{i=4}^{n-1}\int_0^1(1-y_{i}^{n+1-i})y_i^{n-i} dy_i\int_0^11-y_ndy_n \nonumber\\= &\frac{(n-1)!}{u^{n-1}}\frac{u^{n-1}}{n-1}\prod_{i=3}^{n-1}\frac{1}{2(n+1-i)}\frac{1}{2} \nonumber \\ &\hspace{1cm}- \frac{(n-1)!}{u^{2n-1}}\Big(\frac{u^{2(n-1)}}{2(n-1)(n-2)}-\frac{u^{2(n-1)}}{2(n-2)2(n-1)}\Big)\prod_{i=4}^{n-1}\frac{1}{2(n+1-i)}\frac{1}{2} \nonumber \\ =& \frac{1}{2^{n-2}}- \frac{1}{2^{n-1}}=\frac{1}{2^{n-1}}. \nonumber
\end{align}


\end{proof}

\end{document}